\documentclass[11pt,a4paper,dvipsnames,eqrno]{article}
\usepackage[utf8]{inputenc}

\title{Intersection Bodies of Polytopes: Translations and Convexity}
\author{Marie-Charlotte Brandenburg \and Chiara Meroni}

\date{}

\usepackage{a4wide}
\usepackage{amssymb,mathrsfs}
\usepackage{amsmath}
\usepackage{euscript}
\usepackage{amsthm}
\usepackage{mathtools}
\usepackage{amsopn}
\usepackage{stackrel}
\usepackage[all]{xy}
\usepackage{setspace}
\usepackage{float}
\usepackage{blkarray}
\usepackage[dvipsnames]{xcolor}
\usepackage[pdftex, colorlinks, linkcolor=blue, citecolor=blue, urlcolor=blue]{hyperref}
\hypersetup{linkcolor=MidnightBlue, citecolor=magenta, urlcolor=MidnightBlue}
\usepackage{pgfplots}
\usepackage[margin=1cm]{caption}
\usepackage{subcaption}
\pgfplotsset{width=7cm, compat=1.10}
\usepgfplotslibrary{fillbetween}
\usepackage[ruled,vlined]{algorithm2e}
\usepackage{tikz}
\usepackage{tikz-cd}
\usetikzlibrary{matrix}
\usetikzlibrary{intersections}
\usepackage{mathdots}
\usepackage[nameinlink]{cleveref}
\usepackage{verbatim}
\usepackage{enumerate}

\usepackage{algorithmic}

\usepackage[backend=bibtex,bibencoding=utf8,style=alphabetic]{biblatex}
\AtBeginBibliography{\small}

\definecolor{mygreen}{RGB}{37,170,37}
\definecolor{myteal}{RGB}{6,144,179}
\definecolor{myblue}{RGB}{59,100,189}
\definecolor{myorange}{RGB}{240,153,85}
\definecolor{myred}{RGB}{190,15,15}
\definecolor{myyellow}{RGB}{255,220,130}
\definecolor{mypurple}{RGB}{125,0,125}

\colorlet{figblue}{MidnightBlue!40}
\colorlet{figorange}{orange!50}
\colorlet{figgreen}{LimeGreen!60}

\usetikzlibrary{arrows.meta}

\usepackage{todonotes}

\newcommand{\R}{\mathbb{R}}

\newcommand{\conv}{\text{conv}}

\DeclareMathOperator{\vol}{vol}
\DeclareMathOperator{\sgn}{sgn}

\newtheorem{theorem}{Theorem}[section]
\newtheorem*{theorem*}{Results}

\newtheorem{lemma}[theorem]{Lemma}
\newtheorem{proposition}[theorem]{Proposition}

\theoremstyle{definition}

\newtheorem{example}[theorem]{Example}

\newtheorem{remark}[theorem]{Remark}

\newcommand{\linearr}[1]{\mathscr L(#1)}
\newcommand{\cocirc}[2]{s(#1)_{#2}}
\newcommand{\inner}[1]{ \langle #1 \rangle }

\DeclareMathOperator{\aff}{aff}

\newcommand{\ma}{\begin{pmatrix}}
\newcommand{\trix}{\end{pmatrix}}
\newcommand{\sma}{\left(\begin{smallmatrix}}
\newcommand{\strix}{\end{smallmatrix}\right)}
\newcommand{\origin}{\mathbf{0}}
\newcommand{\convof}[1]{\conv \left( #1 \right)}
\DeclareMathOperator{\cone}{cone}
\newcommand{\coneof}[1]{\cone \left( #1 \right)}
\DeclareMathOperator{\vertices}{vert}

\newcommand{\hyparr}{\mathcal H}
\newcommand{\hyparrofX}[1]{\hyparr(#1)}
\newcommand{\hyparrofP}{\hyparr(P)}
\newcommand{\midvert}{\ \middle\vert \ }

\newcommand{\set}[1]{\ensuremath{ \left\lbrace  #1 \right\rbrace }}
\newcommand{\edge}{e}

\colorlet{figbluedark}{black!10!MidnightBlue!80} 
\colorlet{figorangedark}{black!30!orange!70} 
\colorlet{figgreendark}{black!30!LimeGreen} 

\usepackage{parskip}
\allowdisplaybreaks

\makeatletter
\def\keywords{\xdef\@thefnmark{}\@footnotetext}
\makeatother

\makeatletter
\def\mscclasses{\xdef\@thefnmark{}\@footnotetext}
\makeatother

\begin{document}

\maketitle

\mscclasses{MSC classes:
52A30, 
52C35, 
52A38, 
52B11, 
14P10 
}

\begin{abstract}
    \noindent We continue the study of intersection bodies of polytopes, focusing on the behavior of $IP$ under translations of $P$.
    We introduce an affine hyperplane arrangement and show that the polynomials describing the boundary of $I(P+t)$ can be extended to polynomials in variables $t\in \mathbb{R}^d$ within each region of the arrangement. 
    In dimension $2$, we give a full characterization of those polygons such that their intersection body is convex. We give a partial characterization for general dimensions.
\end{abstract}

\vspace{0.2cm}
\section{Introduction}
In the field of convex geometry, intersection bodies have been widely studied from an analytical viewpoint, and mainly in the context of volume inequalities. Originally introduced by Lutwak \cite{lutwak88_intersectionbodiesdual}, they have played a significant role in solving the Busemann-Petty problem, which asks to compare the volume of two convex bodies based on the volumes of their linear sections \cite{Gardner19942,Gardner19941,Koldobsky1998,GKS1999,Zhang1999}.
Unlike its more famous counterparts, the projection body, the intersection body $IK$ of a star body $K$ is not invariant under affine translation. Furthermore, an intersection body can be both convex and non-convex.
Convexity is certified \emph{Busemann's theorem} \cite{busemann49_theoremconvexbodies}, 
which states that $IK$ is convex if $K$ is a convex body centered at the origin (i.e., $K$ is centrally symmetric, where the center of symmetry is the origin), and this statement has been generalized to $L_p$-intersection bodies \cite{berck09_convexityl_pintersection}. On the other hand, given a convex body $K\subseteq \R^d$, there always exists some $t\in \R^d$ such that $I(K+t)$ is not convex \cite[Thm. 8.1.8]{G2006}. 

The occurrence of non-convex intersection bodies has motivated considerations of various measures for capturing the magnitude of their non-convexity, leading to the study of \emph{$p$-convexity} of intersection bodies both over the complex numbers and over the reals \cite{kim11_geometrypconvex,huang12_busemanntheoremcomplex}.
Another direction of research concerns an adaptation of the construction of intersection bodies in order to get convexity, which resolves in \emph{convex intersection bodies} \cite{MeyRes:ConvexIntBody,Stephen:DetProblemConvexIB}. A different relative of intersection bodies is the cross-section body \cite{Martini:CrossSectionBodies,Horst1994}; however, this starshaped set turned out to be non-convex as well, in the general case \cite{Brehm:NonConvexCrossSectionBody}.
Summarizing, many of the positive results towards convexity in all these works concern intersection bodies of centrally symmetric star bodies. In contrast, we focus on affine translates, and consider objects which are not necessarily centrally symmetric.

The goal of this article is to investigate the behavior of intersection bodies of polytopes under translations, and to determine under which translations the intersection body is convex.
In our previous work \cite{BBMS:IntersectionBodiesPolytopes} we exhibit rich semialgebraic structures of intersection bodies of polytopes. However, in general, the intersection body $IP$ of a polytope $P$ is not a basic semialgebraic set, and there exists a central hyperplane arrangement which describes the regions in which the topological boundary of $IP$ is defined by a fixed polynomial.
Taking advantage of these combinatorial and semialgebraic structures opens up new possibilities to study the question of convexity in the present work. In particular, exploiting this semialgebraicity, we are able to characterize convexity by using elementary geometric arguments.

In this article we introduce an \emph{affine hyperplane arrangement} associated to a fixed polytope $P$. We prove that for translation vectors $t \in \R^d$ within a region of this arrangement the polynomials defining the boundary of $I(P+t)$ can be extended to polynomials in $t_1,\dots,t_d$ (\Cref{th:continuous}).
In dimension $2$, we give a full characterization of those polygons with a convex intersection body. We give a partial characterization for general dimension.

\begin{theorem*}
    Let $P$ be a full-dimensional polytope in $\R^d$. 
    \begin{enumerate}[\textup{(}i\textup{)}]
        \item If $d=2$ then $IP$ is convex if and only if $P=-P$.
        \item If $P\subset \R^d$ is a parallelepiped, then $IP$ is convex if and only if $P=-P$.
        \item If $IP$ is strictly convex then $I(P+t)$ is strictly convex, for small translation vectors $t$.
    \end{enumerate}
\end{theorem*}
A full classification of the $2$-dimensional case is given in \Cref{thm:IP_convex_dim2}, and the remaining statements can be found in \Cref{prop:cube} and \Cref{prop:open-ball}. An example of a strictly convex intersection body is given in \Cref{ex:strictly-convex}. 

\textbf{Overview.} 
The article is structured as follows. In \Cref{sec:preliminaries} we review the main concepts and notation from \cite{BBMS:IntersectionBodiesPolytopes}. In \Cref{sec:affine-arrangement} we introduce an affine hyperplane arrangement and describe how it governs the behavior of $IP$ under translation of $P$. 
We then turn to the characterization of convexity, where \Cref{sec:convexity} concerns the $2$-dimensional case, and \Cref{sec:higher_dim} the case of general dimensions. 

\textbf{Acknowledgements.} We are thankful to Christoph Hunkenschröder for posing a question during a seminar discussion which inspired this work. We thank Andreas Bernig and Jesús De Loera for inspiring conversations about intersection bodies and convexity. We thank Isabelle Shankar for helpful feedback that helped us improve our manuscript. We are thankful to the organizers of the conference ``Geometry meets Combinatorics in Bielefeld'', where most of our ideas fell into place. Marie-Charlotte Brandenburg was funded by the Deutsche Forschungsgemeinschaft (DFG, German Research Foundation) – SPP 2298.

\section{Preliminaries}
\label{sec:preliminaries}

We will rely on methods and results which were developed in \cite{BBMS:IntersectionBodiesPolytopes}. In this section we review the most important concepts and results we are going to make use of.

Let $P \subseteq \R^d$ be a convex polytope. The \emph{intersection body} $IP$ of $P$ is the starshaped set
\[
    IP = \set{x \in \R^d \midvert \rho_{IP}(x) \geq 1},
\]
where the \emph{radial function} $\rho_{IP}:\R^d\to \R$ of $IP$ is
\[
    \rho_{IP}(x) = \frac{1}{\|x \|} \vol_{d-1}(P \cap x^\perp).
\]
Here, $\vol_{d-1}$ denotes the $(d-1)$-dimensional Euclidean volume, and $x^\perp \subseteq \R^d$ denotes the linear hyperplane which is orthogonal to $x \in \R^d$, namely the set $x^\perp = \{y\in \R^d \,|\, \langle x, y \rangle = 0\}$.
To obtain meaningful results, we may thus assume that $P \subseteq \R^d$ is a $d$-dimensional polytope throughout this article.
The topological boundary of the intersection body $IP$ is defined by the equation $\partial IP = \{ x \in \R^d \mid \rho_{IP}(x) = 1 \}$. 
Since the radial function satisfies $\rho_{IP}(\lambda x) = \frac{1}{\lambda} \rho_{IP}(x)$ for every $\lambda >0$, it is completely determined by its restriction to the unit sphere.

The intersection body $IP$ of a polytope is governed by the central hyperplane arrangement 
\[
    \hyparrofP = \bigcup_{\substack{v\neq \origin \text{ is a } \\ \text{vertex of } P}} v^\perp.
\]
We denote the set of vertices of $P$ by $\vertices(P)$, and the origin is denoted by $\origin \in \R^d$.
An \emph{open chamber} $C$ of $\hyparrofP$ is a connected component of $\R^d \setminus \hyparrofP$. Given such a chamber $C$, all hyperplanes $x^\perp$ for $x \in C$ intersect $P$ in the interiors of a fixed set of edges.
The radial function restricted to such a chamber is a quotient of polynomials
\begin{equation}\label{eq:radial_chamber}
     \rho_{IP}|_C = \frac{p_C}{\| x \|^2 q_C},
\end{equation}
where $p_C$ is divisible by $\|x\|^2$. Therefore, the topological boundary $\partial IP \cap C$ is the zero-set of the (irreducible) polynomial $\frac{p_C}{\| x \|^2} - q_C$. 
We repeat a key argument in the proof of \eqref{eq:radial_chamber}. 
Let $x \in C$ and $Q = P \cap x^\perp$. The value $\rho_{IP}(x)$ is by definition the volume of $Q$. This computation is done by considering a triangulation $\mathcal T$ of the boundary of $Q$. We extend this to a covering of $\conv(Q,\origin)$ by considering the set $\conv(\Delta,\origin)$ for every simplex $\Delta \in \mathcal T$ such that $\origin \not \in \Delta$. Note that if $\origin \in P$, then this induces a central triangulation of $Q$.
Denoting $v_1, \dots, v_d$ the vertices of a simplex $\Delta\in\mathcal{T}$, the volume of $\conv(\Delta,\origin) = \conv(v_1, \dots, v_d, \origin)$ is, up to a constant scaling factor, given by the determinant of the matrix
\[
    M_{\Delta}(x) = \begin{bmatrix}
		\frac{\langle b_{i_1}, x \rangle a_{i_1} - \langle a_{i_1}, x \rangle b_{i_1}}{\langle b_{i_1}-a_{i_1}, x \rangle} \\
		\vdots \\
		\frac{\langle b_{i_{d-1}}, x \rangle a_{i_{d-1}} - \langle a_{i_{d-1}}, x \rangle b_{i_{d-1}}}{\langle b_{i_{d-1}}-a_{i_{d-1}}, x \rangle} \\
		x
	\end{bmatrix},
\]
where the vertices $v_i$ arise as intersection of $x^\perp$ with edges of $P$, i.e., $v_i = \conv(a_i,b_i) \cap x^\perp$ for $a_i, b_i \in \vertices(P)$. 
Assigning $\sgn(\Delta) \in \{-1,1\}$ to each simplex, this gives 
\[
    \rho_{IP}(x) = \vol_{d-1}(Q) = \frac{1}{\|x\|^2 (d-1)!} \sum_{\Delta \in \mathcal T} \sgn(\Delta) \det(M_\Delta(x)) .
\]

\section{Translations and Affine Hyperplane Arrangements}
\label{sec:affine-arrangement}
Let $P \subseteq \R^d$ be a polytope. In this section we consider how the intersection body of $P+t$ transforms under variation of $t \in \R^d$.
Recall from \Cref{sec:preliminaries} that the combinatorial structure of the boundary of $I(P+t)$ is described by the \emph{central} hyperplane arrangement $\hyparrofX{P+t}$. We thus begin by studying the behavior of this hyperplane arrangement under translation of $P$. For this, we introduce a new \emph{affine} hyperplane arrangement $\linearr{P}$, which captures the essence of $\hyparrofX{P+t}$ under variation of $t$. We show that within a region $R$ of $\linearr{P}$ the polynomials describing the boundary of $I(P+t)$, for $t \in R$, can be extended to polynomials in the variables $t_1,\dots,t_d$.

Let $P \subseteq \R^d$ be a polytope and let $\vertices(P)$ be the set of its vertices. Denote by $H_v = v^\perp \subseteq \R^d$ the hyperplane though the origin that is orthogonal to a vertex $v\in\vertices(P)$.
As described in the previous section, the collection of all such hyperplanes forms a central hyperplane arrangement $\hyparrofP$ in $\R^d$.
For each such hyperplane we define its \emph{positive} and \emph{negative side} as
\[
H_v^+ = \set{ x \in \R^d \mid \inner{x,v} > 0 } \text{ and }
H_v^- = \set{ x \in \R^d \mid \inner{x,v} < 0 }.
\]
We now choose a translation vector $t \in \R^d$ and consider the vertices $\set{v+t \midvert v \in \vertices(P)}$ of the translated polytope $P+t$. The hyperplane arrangement $\hyparrofX{P+t}$ is given by the hyperplanes $(v+t)^\perp$, where $v$ ranges over the vertices of $P$. 
The hyperplane $H_{v+t}$ can be obtained from $H_v$ by a rotation $r_{v,t}: \R^d \to \R^d$ such that $ r_{v,t}\left(\tfrac{v}{||v||}\right) = \tfrac{v+t}{||v+t||} $, and thus $r_{v,t}(H_v) = H_{v+t}$, $r_{v,t}(H^+_v) = H^+_{v+t}$ and $r_{v,t}(H^-_v) = H^-_{v+t}$. 

We label each maximal chamber $C$ of $\hyparrofX{P+t}$ with a sign vector $\cocirc{C}{} \in \{+,-\}^{\vertices(P+t)}$ indexed by the vertices $w=v+t$ of $P+t$, where 
\begin{align*}
	\cocirc{C}{w} &= + \quad\hbox{ if } C \subseteq H^+_w, \\
	\cocirc{C}{w} &= - \quad\hbox{ if } C \subseteq H^-_w.
\end{align*}
The set $\set{\cocirc{C}{} \mid C \hbox{ maximal chamber of } \hyparrofX{P+t} }$ describes the \emph{chirotope} or \emph{signed cocircuits of the underlying oriented matroid} of the hyperplane arrangement \cite[Chapter 6.2.3]{goodman18_handbookdiscretecomputational}.
 
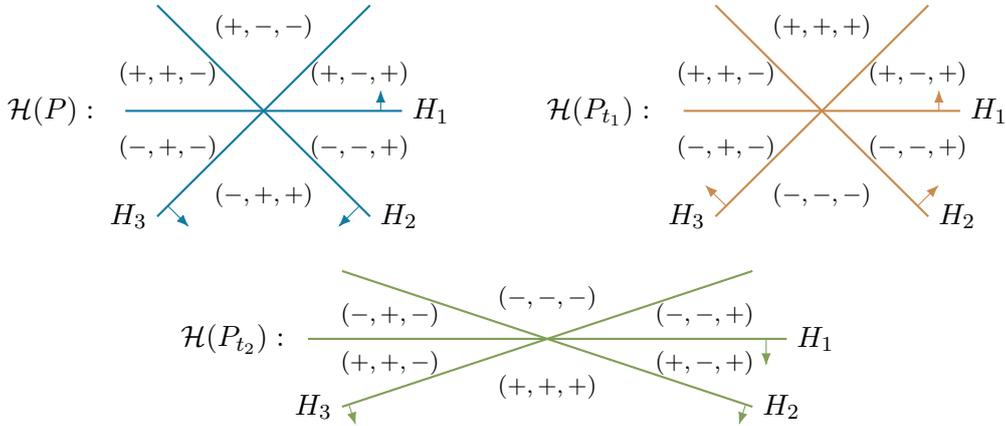
\begin{figure}[!h]
		\centering
		\begin{tikzpicture}[scale=1.4]
	\draw[thick, color=figbluedark] (-1.3,0) -- (1.3,0);
	\draw[thick, color=figbluedark] (-1,1) -- (1,-1);
	\draw[thick, color=figbluedark] (-1,-1) -- (1,1);
	\node[anchor = west] at (1.3,0) {$H_1$};
	\draw[>=Latex, ->, color=figbluedark] (1.1,0) -- (1.1, 0.2);
	\node[anchor = west] at (1,-1) {$H_2$};
	\draw[>=Latex, ->, color=figbluedark] (0.9,-0.9) -- (0.7, -1.1);
	\node[anchor = east] at (-1,-1) {$H_3$};
	\draw[>=Latex, ->, color=figbluedark] (-0.9,-0.9) -- (-0.7, -1.1);
	\node at (0.9,0.35) {\footnotesize{$(+,-,+)$}};
	\node at (0,0.8) {\footnotesize{$(+,-,-)$}};
	\node at (-0.9,0.35) {\footnotesize{$(+,+,-)$}};
	\node at (-0.9,-0.35) {\footnotesize{$(-,+,-)$}};
	\node at (0,-0.8) {\footnotesize{$(-,+,+)$}};
	\node at (0.9,-0.35) {\footnotesize{$(-,-,+)$}};
	\node at (-2,0) {$\hyparrofX{P}:$};
\end{tikzpicture}
\hspace*{2em}
\begin{tikzpicture}[scale=1.4]
	\draw[thick, color=figorangedark] (-1.3,0) -- (1.3,0);
	\draw[thick, color=figorangedark] (-1,1) -- (1,-1);
	\draw[thick, color=figorangedark] (-1,-1) -- (1,1);
	\node[anchor = west] at (1.3,0) {$H_1$};
	\draw[>=Latex, ->, color=figorangedark] (1.1,0) -- (1.1, 0.2);
	\node[anchor = west] at (1,-1) {$H_2$};
	\draw[>=Latex, ->, color=figorangedark] (0.9,-0.9) -- (1.1, -0.7);
	\node[anchor = east] at (-1,-1) {$H_3$};
	\draw[>=Latex, ->, color=figorangedark] (-0.9,-0.9) -- (-1.1, -0.7);
	\node at (0.9,0.35) {\footnotesize{$(+,-,+)$}};
	\node at (0,0.8) {\footnotesize{$(+,+,+)$}};
	\node at (-0.9,0.35) {\footnotesize{$(+,+,-)$}};
	\node at (-0.9,-0.35) {\footnotesize{$(-,+,-)$}};
	\node at (0,-0.8) {\footnotesize{$(-,-,-)$}};
	\node at (0.9,-0.35) {\footnotesize{$(-,-,+)$}};
	\node at (-2.1,0) {$\hyparrofX{P_{t_1}}:$};
\end{tikzpicture}
\vspace*{1em}\\
\begin{tikzpicture}[scale=0.9]
	\draw[thick, color=figgreendark] (-3.5,0) -- (3.5,0);
	\draw[thick, color=figgreendark] (-3,1) -- (3,-1);
	\draw[thick, color=figgreendark] (-3,-1) -- (3,1);
	\node[anchor = west] at (3.5,0) {$H_1$};
	\draw[>=Latex, ->, color=figgreendark] (3.2,0) -- (3.2, -0.4);
	\node[anchor = west] at (3,-1) {$H_2$};
	\draw[>=Latex, ->, color=figgreendark] (2.9,-0.97) -- (2.8, -1.27);
	\node[anchor = east] at (-3,-1) {$H_3$};
	\draw[>=Latex, ->, color=figgreendark] (-2.9,-0.97) -- (-2.8, -1.27);
	\node at (2.3,0.35) {\footnotesize{$(-,-,+)$}};
	\node at (0,0.6) {\footnotesize{$(-,-,-)$}};
	\node at (-2.3,0.35) {\footnotesize{$(-,+,-)$}};
	\node at (-2.3,-0.35) {\footnotesize{$(+,+,-)$}};
	\node at (0,-0.7) {\footnotesize{$(+,+,+)$}};
	\node at (2.3,-0.35) {\footnotesize{$(+,-,+)$}};
	\node at (-4.6,0) {$\hyparrofX{P_{t_2}}:$};
\end{tikzpicture}
		\caption[Hyperplane arrangements of translations of $P$]{The hyperplane arrangements of $P+t$ from \Cref{ex:ordered-types-triangle}.}
		\label{fig:ordered-types-triangle}
\end{figure}
\begin{example}
	\label{ex:ordered-types-triangle}
 	Let $P = \convof{v_1, v_2, v_3}$ be the triangle with vertices
	\[
	v_1 = \ma 0 \\ 1 \trix, \qquad v_2 = \ma -1 \\ -1 \trix, \qquad v_3 =  \ma 1 \\ -1 \trix.
	\]
	\Cref{fig:ordered-types-triangle} shows the hyperplane arrangements $\hyparrofX{P+t}$ for $t_0 = \origin$, $t_1 = (0,2)$, and $t_2 = (0,-2)$.
 	Note that the underlying oriented matroids of $\hyparrofX{P+t}$ for $t = t_1$ and $t = t_2$ are the same.
	We continue with this in \Cref{ex:affine-arrangement-triangle}.
\end{example}

We begin by showing that the signed cocircuit $s(C)$ of a chamber $C$ fully determines the set of edges of $P$ which are intersected by $x^\perp$ for any $x \in C$.

\begin{lemma}\label{lem:intersect-fixed-edges}
	Let $P \subseteq \R^d$ be a polytope and let $t \in \R^d$. Let $C$ be a maximal open chamber of $\hyparrofX{P}$, and $C_t$ be a maximal open chamber of $\hyparrofX{P+t}$ such that $\cocirc{C}{} = \cocirc{C_t}{}$, i.e., their signed cocircuits agree. Given $x \in C, \ x_t \in C_t$ consider 
	\begin{align*}
		\mathcal E &= \{ \edge \subseteq P \mid \edge \text{ is an edge of } P,\ x^\perp \cap \edge \neq \emptyset \}, \\ 
		\mathcal E_t &= \{ \edge_t \subseteq P+t \mid \edge_t \text{ is an edge of } P+t,\ x_t^\perp \cap \edge_t \neq \emptyset \}.
	\end{align*}
	Then $\mathcal E_t = \{e + t \mid e \in \mathcal E\}$.
\end{lemma}
\begin{proof}
	Let $\edge = \convof{v_1, v_2} \in \mathcal E$ be an edge of $P$. Since $x^\perp \cap \edge \neq \emptyset$, we have that $v_1,v_2$ lie on different sides of $x^\perp$. Equivalently, we have $\cocirc{C}{v_1} = -\cocirc{C}{v_2}$, and without loss of generality $\cocirc{C}{v_1} = +$. Thus, $x \in H^+_{v_1} \cap H^{-}_{v_2}$. Since $\hyparrofX{P+t}$ is obtained from $\hyparrofX{P}$ by rotating the hyperplanes individually, and $\cocirc{C}{} = \cocirc{C_t}{}$, it follows that $x_t \in H^+_{v_1+t} \cap H^{-}_{v_2+t}$. Since $\edge+t$ is an edge of $P+t$ if and only if $\edge$ is an edge of $P$, the claim follows.
\end{proof}

We consider the \emph{affine hyperplane arrangement}
\[
\linearr{P} = \{ \aff(-v_1,\dots,-v_d) \mid v_1,\dots,v_d \text{ are affinely independent vertices of } P \},
\]
where $\aff(-v_1,\dots,-v_d)$ denotes the unique affine hyperplane containing the points $-v_1,$ $\dots,$ $-v_d$.
An \emph{open region} $R$ of $\linearr{P}$ is a connected component of $\R^d \setminus \linearr{P}$.
We emphasize that there are two hyperplane arrangements in $\R^d$ which which we consider simultaneously. We have the \emph{central} hyperplane arrangement $\hyparrofX{P+t}$, which depends on the choice of $t$, and subdivides $\R^d$ into open $d$-dimensional cones, which we call \emph{chambers} of $\hyparrofX{P+t}$. On the other hand, we have the \emph{affine} hyperplane  arrangement $\linearr{P}$, which subdivides $\R^d$ into open $d$-dimensional components, which we call \emph{regions} of $\linearr{P}$. Note that $\linearr{P+t} = \linearr{P} - t$ by construction.

\begin{example}\label{ex:affine-arrangement-triangle}
	Let $P$ be the triangle from \Cref{ex:ordered-types-triangle}. The affine line arrangement $\linearr{P}$ is shown in \Cref{fig:affine-arrangement-triangle}. Note that the translation vectors $t = t_0,t_1,t_2$ all lie in different regions of the arrangement, despite the fact that the signed cocircuits of $\hyparrofX{P+t_1}$ and $\hyparrofX{P+t_2}$ agree, as displayed in \Cref{fig:ordered-types-triangle}.
	\begin{figure}[ht]
		\centering
		\begin{tikzpicture}[scale=0.6]
	\draw[thick] (-2, 3) -- (1,-3);  
	\draw[thick] (2, 3) -- (-1,-3);
	\draw[thick] (-3,1) -- (3,1);
	\filldraw[color=figbluedark] (0,0) circle (2pt);
	\node[anchor=south] at (0,0) {$t_0$};
	\filldraw[color=figorangedark] (0,2) circle (2pt);
	\node[anchor=south] at (0,2) {$t_1$};
	\filldraw[color=figgreendark] (0,-2) circle (2pt);
	\node[anchor=north] at (0,-2) {$t_2$};
	\node[anchor = south east] at (-1,1) {$-v_3$};
	\node[anchor = south west] at (1,1) {$-v_2$};
	\node[anchor = west] at (0,-1) {$-v_1$};
\end{tikzpicture}
		\caption{The arrangement $\linearr{P}$ of the triangle from \Cref{ex:ordered-types-triangle,ex:affine-arrangement-triangle}.}
		\label{fig:affine-arrangement-triangle}
	\end{figure}
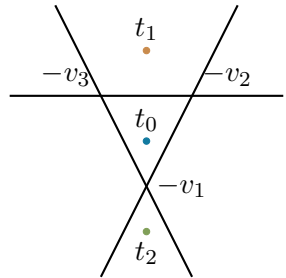
\end{example}

In the following we show that $\linearr{P}$ captures the characteristics of $\hyparrofX{P+t}$ under variation of $t$. More precisely, we show that within a region $R$ of $\linearr{P}$ the polynomials describing the boundary of $I(P+t)$, for $t \in R$, can be extended to polynomials in $t_1,\dots,t_d$.

\begin{proposition}
    \label{prop:cocircuits-fixed}
    Let $P \subseteq \R^d$ be a polytope and $R$ be an open region of $\linearr{P}$. Then the set of signed cocircuits of $\hyparrofX{P+t}$ is fixed for all $t \in R$.
\end{proposition}

\begin{proof}
    Let $v_1,\dots,v_d$ be affinely independent vertices of $P$. By construction of $\linearr{P}$, $R$ does not intersect $\mathcal{A} = \aff(-v_1,\dots,-v_d)$, i.e., $R$ is strictly contained in one side of this hyperplane. Without loss of generality, we assume $R \subseteq \mathcal A^+$. The points $w_k = v_k + t$, for $k=1,\ldots,d$, are linearly independent vertices of $P+t$ for all $t \in \R^d \setminus \mathcal A$. Hence, the subarrangement of $\hyparrofX{P+t}$ consisting of hyperplanes $w_1^\perp,\dots,w_d^\perp$ is a simplicial arrangement which dissects $\R^d$ into $2^d$ open chambers, where each chamber is the image of an orthant of $\R^d$ under the linear map $f$ defined by $e_i \mapsto w_i$ for all $i = 1,\ldots,d$. Note that the signed cocircuits are fixed for every $t \in \mathcal A^+$. 
    We now consider $\hyparrofX{P+t}$ as common refinement of all subarrangements formed by $d$ hyperplanes with linearly independent normals. The signed cocircuit of a chamber of $\hyparrofX{P+t}$ is uniquely determined by the signed cocircuits of all subarrangements, and the cocircuits of the subarrangements are fixed for all $t \in R$. Thus, the cocircuits of $\hyparrofX{P+t}$ are fixed for all $t \in R$.
\end{proof}

\begin{theorem}\label{th:continuous}
	Let $R$ be an open region of $\linearr{P}$, $t \in R$, 
    and let $C_t$ be an open chamber of $\hyparrofX{P+t}$.
 Then the radial function $\rho_{I(P+t)}|_{C_t}$ of $I(P+t)$ restricted to the chamber $C_t$ and for $t\in R$ is a polynomial in the variables $t_1,\ldots ,t_d$ of degree at most $d-1$.
\end{theorem}

\begin{proof}
    By \Cref{prop:cocircuits-fixed}, for a fixed region $R$ the set of signed cocircuits of $\hyparrofX{P+t}$ is fixed.
    \Cref{lem:intersect-fixed-edges} then implies that given a region $R$, $t \in R$, and a chamber $C_t$ of $\hyparrofX{P+t}$, for any vector $x \in C_t$ the set of edges of $P+t$ which intersect $x^\perp$ is fixed. Let $Q = P \cap x^\perp$, for a certain $x \in C_t$, and let $\mathcal T$ be a triangulation of $\partial Q$, as explained in \Cref{sec:preliminaries}. Let $\Delta \in \mathcal T$ be a maximal simplex with vertices $v_1, \dots , v_{d-1}$ such that $\origin \not \in \Delta$ and, for each $i = 1,\ldots,d-1$, let $a_i,b_i \in \vertices(P)$ such that $v_i = \conv(a_i +t, b_i +t) \cap x^\perp$.
    The volume of the $d$-dimensional simplex $\conv(\Delta,\origin)$ is, up to a multiplicative factor of $\frac{\pm 1}{\|x\| (d-1)!}$, the determinant of the matrix
    \begin{align*}
    M_{\Delta}(x,t) = \!\begin{bmatrix}
		\frac{\langle b_{1}+t, x \rangle (a_{1}+t) - \langle a_{1}+t, x \rangle (b_{1}+t)}{\langle b_{1}-a_{1}, x \rangle} \\
		\vdots \\
		\frac{\langle b_{d-1}+t, x \rangle (a_{d-1}+t) - \langle a_{d-1}+t, x \rangle (b_{d-1}+t)}{\langle b_{d-1}-a_{d-1}, x \rangle} \\
		x
	\end{bmatrix} 
     = M_{\Delta}(x,\origin) +\! \begin{bmatrix}
        \frac{\langle t, x \rangle (a_{1}-b_1)}{\langle b_{1}-a_{1}, x \rangle} + t \\
		\vdots \\
		\frac{\langle t, x \rangle (a_{d-1}-b_{d-1}) }{\langle b_{d-1}-a_{d-1}, x \rangle} + t \\
		\origin
    \end{bmatrix}.
\end{align*}
The determinant of this matrix is a polynomial in the variables $t_1,\dots,t_d$ of degree at most $d-1$. Since the volume of $Q$ can be computed as
\[
    \vol(Q) = \frac{1}{\|x\| (d-1)!} \sum_{\Delta \in \mathcal T} \sgn(\Delta) \det(M_\Delta(x,t)),
\]
the claim follows.
\end{proof}

\begin{example}\label{ex:affine-lines-square}
	\Cref{fig:affine-lines-square} shows the continuous deformation of the intersection body $I(P+t)$ of the unit square $P = [-1,1]^2$ under translation by $t\in \R^2$ within each bounded region of the affine line arrangement $\linearr{P}$.
\end{example}
\begin{figure}[ht]
	\centering
	\includegraphics[height=3.2in]{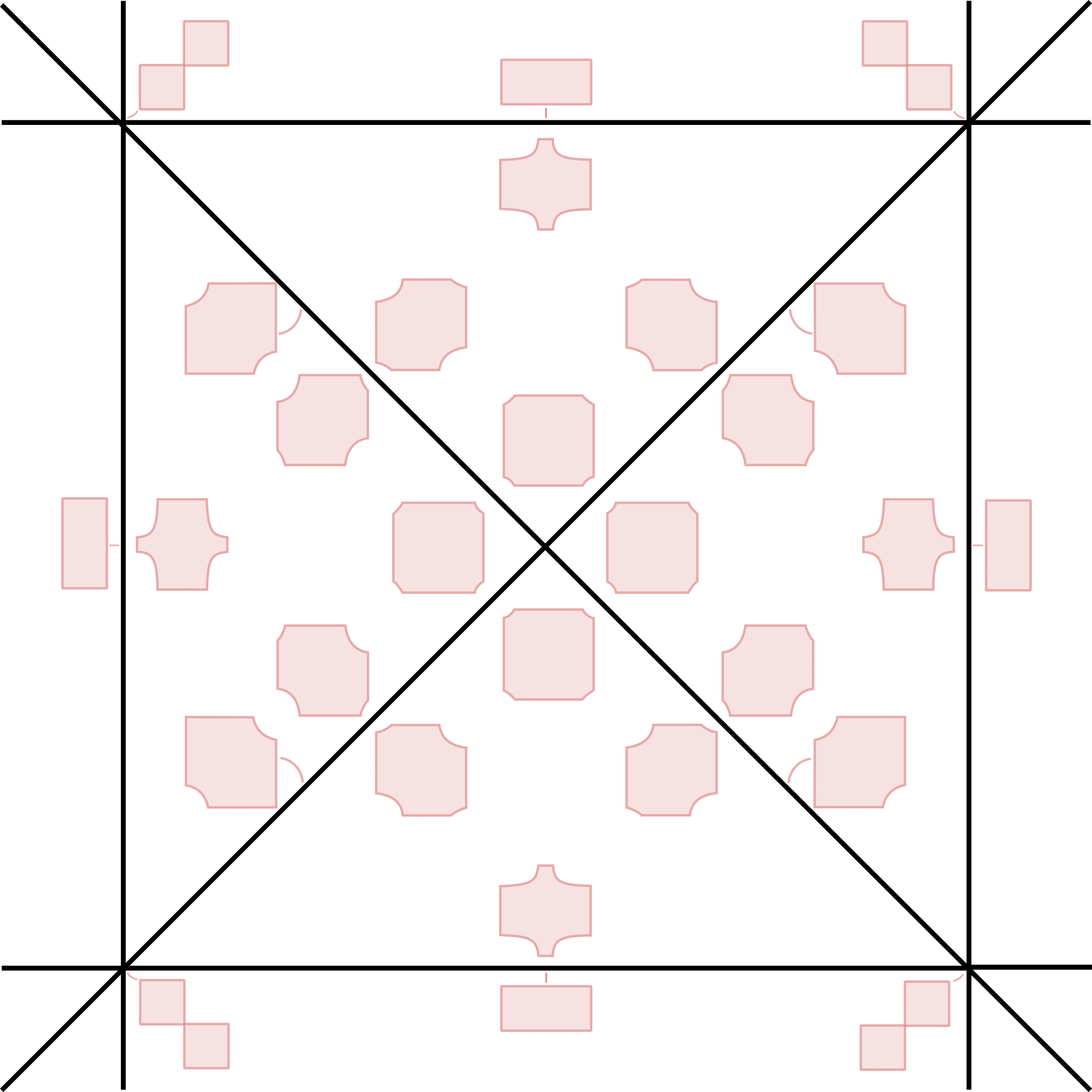}
	\caption{The arrangement $\linearr{P}$ of affine lines for $P = [-1,1]^2$, in black, together with $I(P+t)$ for different choices of $t$, in red, as in \Cref{ex:affine-lines-square}.}
	\label{fig:affine-lines-square}
\end{figure}

\section{Convexity in Dimension 2}\label{sec:convexity}

For each fixed region $R$ of the affine line arrangement $\linearr{P}$, \Cref{th:continuous} implies that, as we move $t \in R$ continuously, the intersection body $I(P+t)$ deforms continuously as well. We now characterize under which circumstances the intersection body of a \emph{polygon} is convex. Note that $IP$ cannot be convex if the origin lies outside of $P$ or is a vertex of $P$ (the argument for general dimensions will be given in \Cref{rmk:lower-dim-faces}). 
We thus consider the distinct cases of when the origin lies in the interior of $P$, and when the origin lies in the interior of an edge.
\Cref{fig:affine-lines-square} indicates that in the case of the square, the intersection body of $P+t$ is convex for precisely $5$ translation vectors: the center of symmetry, as well as the midpoints of the four edges. In \Cref{thm:IP_convex_dim2} we show that the number of such translation vectors is always finite, and that parallelograms maximize this number.

The goal of this section is to give a characterization of polygons whose intersection bodies are convex. In the following \Cref{prop:parallel_segments,prop:non_parallel_non_convex} we consider polygons with the origin in the interior, and characterize the geometry of the boundary of $IP$. More precisely, we will see that the chambers in which $IP$ is convex correspond to pairs of parallel edges of $P$, and that the polynomials defining the boundary of $IP$ are linear in this case.

\begin{proposition}\label{prop:parallel_segments}
	Let $P \subseteq \R^2$ be a polygon.
	Let $C$ be a chamber of $\hyparrofX{P}$, and consider $x\in C$. We denote by $v_1(x), v_2(x)$ the points of intersection $x^\perp \cap \partial P = \set{v_1(x), v_2(x)}$.
 Let $\convof{a_1,b_1}, \convof{a_2,b_2}$ be edges of $P$ such that $v_1(x)\in \convof{a_1,b_1}$ and $v_2(x)\in \convof{a_2,b_2}$. Then the polynomial defining $\partial IP$ in the chamber $C$ is linear if and only if the segments $\convof{a_1,b_1}$ and $\convof{a_2,b_2}$ are parallel. 
\end{proposition}

\begin{proof}
	We want to prove that $\{x\in C \mid \  \rho_{IP}|_C(x)=1\}$ is a line segment if and only if the two edges $\convof{a_i,b_i}$ are parallel. Assume that $v_1(x) = \lambda a_1 + (1-\lambda) b_1$ and $v_2(x) = \mu a_2 + (1-\mu) b_2$ for some $\lambda,\mu \in (0,1)$. Since $v_1(x), v_2(x) \in x^\perp$, we have 
	\begin{equation*}
		\lambda = \frac{\inner{b_1, x}}{\inner{b_1-a_1, x}}, \qquad \mu = \frac{\inner{b_2, x}}{\inner{b_2-a_2, x}}.
	\end{equation*}

	We compute the length of $\conv(v_1(x),v_2(x))$, or equivalently of $\conv(\origin, v_1(x) - v_2(x))$. We do this via the area of the triangle with vertices $\origin, v_1(x) - v_2(x)$ and $\tfrac{x}{\| x \|^2}$.
	Hence, the radial function can be computed by
	the determinantal expression
	\[
	\rho_{IP}|_C(x) = \frac{1}{\| x\|^2} \left\vert \det \begin{bmatrix} v_1(x)-v_2(x) \\ x \end{bmatrix} \right\vert.
	\]
	We compute the radial function explicitly. First,
	\begin{equation*}
		v_1(x)-v_2(x) = \frac{ \left( \inner{b_2-a_2, x} \left(\inner{b_1,x} a_1
			- \inner{a_1,x} b_1\right) - \inner{b_1-a_1, x} \left(\inner{b_2,x} a_2
			- \inner{a_2,x} b_2 \right) \right)}{\inner{b_1-a_1, x} \inner{b_2-a_2, x}}.
	\end{equation*}
	The boundary $\partial P \cap C$ is given by the set of points $x \in C$ such that $\rho_{IP}|_C(x)=1$, i.e., the points which satisfy
	{\small{
	\begin{equation}\label{eq:ab,cd}
		\begin{multlined}
			\frac{1}{\| x\|^2}  \det \begin{bmatrix}
			 \inner{b_2-a_2, x} \left(\inner{b_1,x} a_1 - \inner{a_1,x} b_1\right) - \inner{b_1-a_1, x} \left(\inner{b_2,x} a_2 - \inner{a_2,x} b_2 \right)  \\ 
			x 
			\end{bmatrix} \\
			\qquad = \; \inner{b_1-a_1, x} \inner{b_2-a_2, x},
		\end{multlined}
	\end{equation}
}}

	assuming that the determinant in the left hand side is positive in $C$ (otherwise it gets multiplied by $-1$). This determinant is a cubic polynomial in $x$, which by \cite[Prop. 5.5]{BBMS:IntersectionBodiesPolytopes} is divisible by $\|x\|^2$. Hence, the left hand side of \eqref{eq:ab,cd} is a homogeneous linear polynomial in $x$. 
	It divides the right hand side if and only if $(b_2-a_2) = \kappa (b_1-a_1)$ for some $\kappa \in \R$, i.e., if the the two edges $\convof{a_i,b_i}$ are parallel. In this case \eqref{eq:ab,cd} is a linear equation, and hence the curve defined by \eqref{eq:ab,cd} is a line; otherwise it is a conic, passing through the origin.
\end{proof}

\begin{proposition}\label{prop:non_parallel_non_convex}
	Let $P\subseteq \R^2$ be polygon with the origin in its interior. If there exists a line through the origin which intersects $\partial P$ in two non-parallel edges, then $IP$ is not convex.
\end{proposition}

\begin{proof}
	Let $C$ be a chamber of of $\hyparrofX{P}$ such that $x^\perp$ intersects two non-parallel edges $\ell_1, \ell_2$ of $P$. 
	Consider $u_a,u_b\in C\cap S^1$. As shown in \Cref{fig:non_parallel_non_convex}, we denote 
	\begin{align*}
		& \qquad u_a^\perp \cap \ell_1 = a = \sma a_1 \\ a_2 \strix, && u_b^\perp \cap \ell_1 = b = \sma b_1 \\ b_2\strix, \qquad  \\
		& \qquad u_a^\perp \cap \ell_2 = -\alpha a, && u_b^\perp \cap \ell_2 = -\beta b, \qquad 
	\end{align*}
	for some positive real numbers $\alpha, \beta > 0$. Since $\ell_1$ and $\ell_2$ are not parallel, we have $\alpha \neq \beta$. 
	We can choose $a,b$ such that $u_a = \frac{1}{\|a\|} \sma a_2 \\ -a_1 \strix$ and $u_b = \frac{1}{\|b\|} \sma b_2 \\ -b_1 \strix$. The lengths of the line segments $u_a^\perp \cap P = \convof{a , -\alpha a}$ and $u_b^\perp \cap P = \convof{b, -\beta b}$ are 
	\begin{align*}
		\| u_a^\perp \cap P \|= \| a - ( - \alpha a) \| = (1+\alpha) \|a \|, \\
		\| u_b^\perp \cap P \|= \| b - ( - \beta b) \| = (1+\beta) \|b \|.
	\end{align*}
	Thus, the boundary points of $IP$ in directions $u_a,u_b$ are
	\begin{align*}
	p_a &:= \rho_{IP} (u_a) \: u_a = (1 + \alpha ) \|a\| \: u_a = (1 + \alpha ) \: \ma a_2 \\ -a_1 \trix, \\
	p_b &:= \rho_{IP} (u_b) \: u_b = (1 + \beta ) \|b\| \: u_b = (1 + \beta ) \:  \ma b_2 \\ -b_1\trix 
\end{align*}
	respectively. Consider the midpoint $\frac{a+b}{2}\in \ell_1$ and let $u_{a+b}$ be the unit vector in $C$ orthogonal to $a+b$ (and thus also to $\frac{a+b}{2}$). 
	Then $u_{a+b} = \tfrac{1}{\|a+b\|} \sma a_2 + b_2 \\ -a_1 - b_1\strix$,
	$u_{a+b}^\perp\cap \ell_2 = -\frac{\alpha\beta}{\alpha + \beta} (a+b)$ and the boundary point of $IP$ in direction $u_{a+b}$ is
	\begin{equation*}
		p_{a+b} =  \rho_{IP} (u_{a+b}) \: u_{a+b} = \left(\frac{1}{2} + \frac{\alpha\beta}{\alpha + \beta} \right) \|a+b\| \: u_{a+b} = \left(\frac{1}{2} + \frac{\alpha\beta}{\alpha + \beta} \right) \: \ma a_2 + b_2 \\ -a_1 - b_1 \trix.
	\end{equation*}

	Let $q = \convof{p_a,p_b} \cap \coneof{u_{a+b}}$, as in \Cref{fig:non_parallel_non_convex}.
	We want to prove that $IP\cap C$ is not convex, by showing that $\|q\| > \| p_{a+b}\|$. 
	Indeed, we can compute that 
	\[
	q = \frac{(1+\alpha)(1+\beta)}{2+\alpha+\beta} (a_2 + b_2,-a_1 - b_1),
	\]
	and therefore
	\begin{align*}
		\|q\| - \| p_{a+b}\| &= \frac{(1+\alpha)(1+\beta)}{2+\alpha+\beta} \|a+b\| - \left( \frac{1}{2} + \frac{\alpha\beta}{\alpha + \beta} \right) \|a+b\| \\
		&= \frac{(\alpha-\beta)^2}{2(2+\alpha+\beta)(\alpha+\beta)} \|a+b\|.
	\end{align*}
	Since $\alpha\neq\beta$, this expression is strictly positive, and so $q\not\in IP$. This proves that $p_a,p_b \in IP$, but the segment $\convof{p_a,p_b}$ is not contained in $IP$. Hence, $IP$ is not convex.
 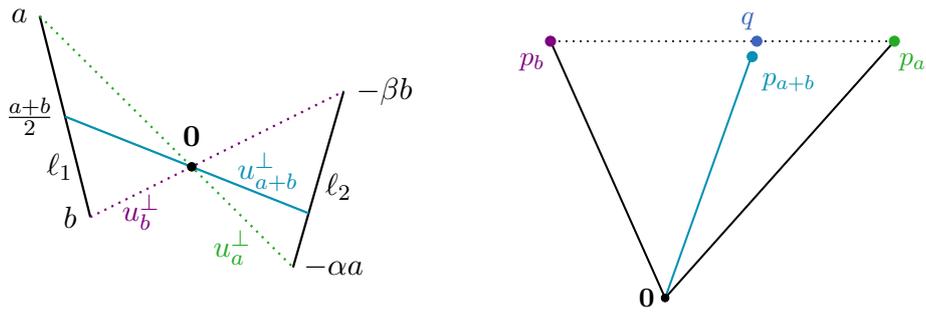
\begin{figure}[!h]
	\centering
	\begin{tikzpicture}[scale=1.05]
	\begin{axis}[
		width=2.2in,
		height=1.9in,
		scale only axis,
		xmin=-2.2,
		xmax=2.2,
		ymin=-1.6,
		ymax=2.2,
		ticks = none, 
		ticks = none,
		axis background/.style={fill=white},
		axis line style={draw=none} 
		]

		\addplot [color=black,thick,solid,forget plot]
		table[row sep=crcr]{%
			-1   -0.5\\
			-1.5   1.5\\
		};
		\addplot [color=black,thick,solid,forget plot]
		table[row sep=crcr]{%
			1.5   0.75\\
			1   -1\\
		};
		
		\addplot [color=mypurple,thick,solid,dotted,forget plot]
		table[row sep=crcr]{%
			-1   -0.5\\
			1.5   0.75\\
		};
		\addplot [color=mygreen,thick,solid,dotted,forget plot]
		table[row sep=crcr]{%
			-1.5   1.5\\
			1   -1\\
		};
		
		\addplot [color=myteal,thick,solid,forget plot]
		table[row sep=crcr]{%
			-1.25   0.5\\
			1.154   -0.462\\
		};
		
		\node[black] (P) at (axis cs:-1.7,1.5) {$a$};
		\node[black] (P) at (axis cs:-1.2,-0.5) {$b$};
		\node[black] (P) at (axis cs:1.4,-1) {$-\alpha a$};
		\node[black] (P) at (axis cs:1.9,0.75) {$-\beta b$};
		\node[black] (P) at (axis cs:-1.6,0.5) {$\frac{a+b}{2}$};
		\node[black] (P) at (axis cs:0,0.3) {$\origin$};
		
		\addplot[only marks,mark=*,mark size=1.5pt,black,
		]  coordinates {
			(0,0) 
		};
		
		\node[black] (P) at (axis cs:-1.3,0) {$\ell_1$};
		\node[black] (P) at (axis cs:1.45,-0.2) {$\ell_2$};
		\node[mygreen] (P) at (axis cs:0.4,-0.8) {$u_a^\perp$};
		\node[mypurple] (P) at (axis cs:-0.5,-0.45) {$u_b^\perp$};
		\node[myteal] (P) at (axis cs:0.75,-0.05) {$u_{a+b}^\perp$};

	\end{axis}
\end{tikzpicture}
	\quad
	\begin{tikzpicture}[scale=0.95]
	\begin{axis}[
		width=2.6in,
		height=1.8in,
		scale only axis,
		xmin=-2.2,
		xmax=3,
		ymax=2.9,
	    ticks = none, 
		ticks = none,
		axis background/.style={fill=white},
		axis line style={draw=none} 
		]
		
		\addplot [color=black,thick,solid,forget plot]
		table[row sep=crcr]{%
			2.5   2.5 \\
			0   0\\
		};
		\addplot [color=black,thick,solid,forget plot]
		table[row sep=crcr]{%
			-1.25   2.5 \\
			0   0\\
		};
		\addplot [color=myteal,thick,solid,forget plot]
		table[row sep=crcr]{%
			0.961   2.4\\
			0   0\\
		};
		\addplot [color=black!80,thick,solid,dotted,forget plot]
		table[row sep=crcr]{%
			2.5   2.5 \\
			-1.25   2.5\\
		};
		
		
		
		\node[mygreen] (P) at (axis cs:2.7,2.3) {$p_a$};
		\node[mypurple] (P) at (axis cs:-1.45,2.3) {$p_b$};
		\node[myblue] (P) at (axis cs:0.9,2.7) {$q$};
		\node[myteal] (P) at (axis cs:1.35,2.1) {$p_{a+b}$};
		\node[black] (P) at (axis cs:-0.2,0) {$\origin$};
		
		\addplot[only marks,mark=*,mark size=1.5pt,black,
		]  coordinates {
			(0,0) 
		};
		\addplot[only marks,mark=*,mark size=2pt,mygreen,
		]  coordinates {
			(2.5,2.5) 
		};
		\addplot[only marks,mark=*,mark size=2pt,mypurple,
		]  coordinates {
			(-1.25,2.5) 
		};
		\addplot[only marks,mark=*,mark size=2pt,myblue,
		]  coordinates {
			(1,2.5) 
		};
		\addplot[only marks,mark=*,mark size=2pt,myteal,
		]  coordinates {
			(0.95,2.35) 
		};
		
	\end{axis}
\end{tikzpicture}
	\caption[The proof of \Cref{prop:non_parallel_non_convex} in a picture]{The proof of \Cref{prop:non_parallel_non_convex} in a picture. Left: the lines orthogonal to $u_a,u_b, u_{a+b}$ and their intersections with the edges $\ell_1, \ell_2$ of $P$. Right: the points $p_a,p_b,p_{a+b}\in \partial IP$, and the point $q\in \convof{p_a,p_b}$, but $q\not \in IP$.}
	\label{fig:non_parallel_non_convex}
\end{figure}
\end{proof}

We are now ready to move towards a full classification of convexity of intersection bodies of polygons for any translation. 
Note that if $P$ is centrally symmetric, then the convexity of $P$ and the description of $IP$ follow from the following classical statement.
\begin{theorem}[{\cite[Theorem 8.1.4]{G2006}}]\label{thm:dim2_convex}
	Let $K \subseteq \R^2$ be a centrally symmetric convex body centered at the origin. Then $IK = r_{\frac{\pi}{2}}(2K)$, where $r_{\frac{\pi}{2}}$ is a counter-clockwise rotation by $\frac{\pi}{2}$.
\end{theorem}
Our goal is to classify also the cases in which $P$ is not centrally symmetric and centered at the origin. 
We now prove the main result of this section.

\begin{theorem}\label{thm:IP_convex_dim2}
	Let $P \subseteq \R^2$ be a polygon. Then $IP$ is a convex body if and only if
	$P=-P$.
\end{theorem}
\begin{proof}
	As noted in \Cref{rmk:lower-dim-faces}, $IP$ is not convex if the origin lies in $\R^2\setminus P$, or if the origin is a vertex of $P$. We are left to analyze the cases in which the origin lies in the interior of $P$ or in the interior of an edge of $P$.
	
	We first consider the case in which the origin lies in the interior of $P$ and show that $IP$ is convex if and only if $P = -P$. If $P=-P$, then \Cref{thm:dim2_convex} implies that $IP$ is convex. Assume now that $IP$ is convex, and the origin lies in the interior of $P$. Then $C\cap IP$ is convex for every chamber $C$ of $\hyparrofX{P}$. In particular, by \Cref{prop:non_parallel_non_convex}, every line $u^\perp$, $u \in S^1$, which does not intersect a vertex of $P$ intersects $\partial P$ in the interior of two parallel edges. Hence, the edges of $P$ come in pairs of parallel edges.
	We rotate $u \in S^1$ continuously.
	Whenever $u^\perp$ crosses a vertex
	of one edge, it must also cross a vertex in the parallel edge, since otherwise this results in a pair of non-parallel edges. This implies that for every vertex $v$ of $P$, there exists a vertex $w$ of $P$ such that $w = -\lambda v$ for some $\lambda >0$. Since all edges are pairwise parallel, this positive scalar $\lambda$ is the same for all vertices. Therefore, we also get that $v = -\lambda w$, which implies that $\lambda=1$. Hence, $P=-P$.

    Consider now the case in which the origin lies in the interior of an edge of $P$ with normal vector $e \in \R^2$. Thus,  $\rho_{IP}(x) = \tfrac{1}{2} \rho_{ I(P \cup -P)}(x)$ for all $x \in \R^2 \setminus \R e$ and $\rho_{IP}(e) > \tfrac{1}{2} \rho_{I(P \cup -P)}(e)$. Here, $\R e$ denotes the line spanned by $e$.
	Since the origin lies in the interior of the star body $P \cup -P$, its radial function is continuous, which implies that also $\tfrac{1}{2}\rho_{ I(P \cup -P)} $ is continuous. Hence, $\rho_{IP}$ is discontinuous, and therefore $IP$ is not convex. 
\end{proof}

\begin{remark}\label{rmk:almost-convex}

The last case of the proof of \Cref{thm:IP_convex_dim2} can be made more precise. 
Using the notion of chordal symmetral from \cite[Chapter 5.1]{G2006}, we deduce that
\begin{equation*}
		\frac{1}{2} I(P\cup -P)  =   P \cup -P.
	\end{equation*}
	Therefore, $I(P\cup -P)$ is convex if and only if $P \cup -P$ is convex.
    This is the case if and only if the origin is the midpoint of an edge
    and the sum of the angles adjacent to this edge is at most $\pi$.
    Using elementary properties of the sums of interior and exterior angles of polygons, it can be shown that a polygon admits at most $4$ such edges, and equality is realized exactly when $P$ is a parallelogram.
    \Cref{fig:convexIP_classification} shows a collection of examples of polygons, together with the possible positions of the origin on such edges.
    In this case, the argument from the proof of \Cref{thm:IP_convex_dim2} implies that the Euclidean closure of $IP \setminus \R e$ is convex. Here, $\R e$ denotes the line spanned by $e$.
\end{remark}

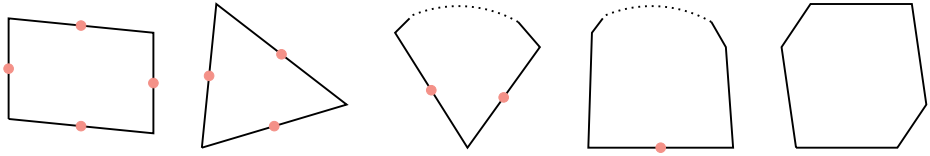
\begin{figure}[h]
	\centering
	\begin{tikzpicture}[scale=0.9]
	\begin{axis}[
		width=1in,
		height=1in,
		scale only axis,
		xmin=-1.2,
		xmax=1.2,
		ymin=-1.2,
		ymax=1.2,
		ticks = none, 
		ticks = none,
		axis background/.style={fill=white},
		axis line style={draw=none} 
		]

		\addplot [color=black,thick,solid,forget plot]
		table[row sep=crcr]{%
			-1   -0.6\\
			-1   0.8\\
			1   0.6\\
			1   -0.8\\
			-1   -0.6\\
		};
		
		
		\addplot[only marks,mark=*,mark size=2pt,Salmon,
		]  coordinates {
			(1,-0.1) (-1,0.1) (0,0.7) (0,-0.7)
		};
		
	\end{axis}
\end{tikzpicture} 
	\begin{tikzpicture}[scale=0.9]
	\begin{axis}[
		width=1in,
		height=1in,
		scale only axis,
		xmin=-1.2,
		xmax=1.2,
		ymin=-1.2,
		ymax=1.2,
		ticks = none, 
		ticks = none,
		axis background/.style={fill=white},
		axis line style={draw=none} 
		]

		\addplot [color=black,thick,solid,forget plot]
		table[row sep=crcr]{%
			-1   -1\\
			-0.8   1\\
			1   -0.4\\
			-1   -1\\
		};
		
		
		\addplot[only marks,mark=*,mark size=2pt,Salmon,
		]  coordinates {
			(-0.9,0) (0.1,0.3) (0,-0.7)
		};
		
	\end{axis}
\end{tikzpicture} 
	\begin{tikzpicture}[scale=0.9]
	\begin{axis}[
		width=1in,
		height=1in,
		scale only axis,
		xmin=-1.2,
		xmax=1.2,
		ymin=-1.2,
		ymax=1.2,
		ticks = none, 
		ticks = none,
		axis background/.style={fill=white},
		axis line style={draw=none} ,
		disabledatascaling
		]

		\addplot [color=black,thick,solid,forget plot]
		table[row sep=crcr]{%
			-0.8   0.8\\
			-1   0.6\\
			0   -1\\
			1   0.4\\
			0.7   0.75\\
		};
		\draw [thick,dotted] (axis cs:0.7,0.75) arc [radius=1.65,start angle=60,end angle=113];
		
		\addplot[only marks,mark=*,mark size=2pt,Salmon,
		]  coordinates {
			(-0.5,-0.2) (0.5,-0.3)
		};
		
	\end{axis}
\end{tikzpicture} 
	\begin{tikzpicture}[scale=0.9]
	\begin{axis}[
		width=1in,
		height=1in,
		scale only axis,
		xmin=-1.2,
		xmax=1.2,
		ymin=-1.2,
		ymax=1.2,
		ticks = none, 
		ticks = none,
		axis background/.style={fill=white},
		axis line style={draw=none} ,
		disabledatascaling
		]

		\addplot [color=black,thick,solid,forget plot]
		table[row sep=crcr]{%
			-0.8   0.8\\
			-0.95   0.6\\
			-1   -1\\
			1   -1\\
			0.9   0.4\\
			0.7   0.75\\
		};
		\draw [thick,dotted] (axis cs:0.7,0.75) arc [radius=1.65,start angle=60,end angle=113];
		
		\addplot[only marks,mark=*,mark size=2pt,Salmon,
		]  coordinates {
			(0,-1)
		};
		
	\end{axis}
\end{tikzpicture} 
	\begin{tikzpicture}[scale=0.9]
	\begin{axis}[
		width=1in,
		height=1in,
		scale only axis,
		xmin=-1.2,
		xmax=1.2,
		ymin=-1.2,
		ymax=1.2,
		ticks = none, 
		ticks = none,
		axis background/.style={fill=white},
		axis line style={draw=none} 
		]

		\addplot [color=black,thick,solid,forget plot]
		table[row sep=crcr]{%
			-0.8   -1\\
			-1   0.4\\
			-0.6   1\\
			0.8   1\\
			1   -0.4\\
			0.6   -1\\
			-0.8   -1\\
		};
		

	\end{axis}
\end{tikzpicture}
	\caption[Examples in which $IP$ is convex]{Examples of positions of the origin (orange bullet) in which $IP$ is almost convex, as described in \Cref{rmk:almost-convex}. From left to right: a parallelogram, an acute triangle, a diamond shape, a panettone shape, and a centrally symmetric polygon (which has no admissible positions).}
	\label{fig:convexIP_classification}
\end{figure}

We close this section by pointing out that many arguments made in this section do not generalize to higher dimensions: In contrast to \Cref{prop:parallel_segments,prop:non_parallel_non_convex}, in higher dimensions there exist convex pieces $IP \cap C$ which are not linear.
Furthermore, the identification with the chordal symmetral body, as in \Cref{rmk:almost-convex}, does not hold in general. However, these insights in the $2$-dimensional case will turn out to be essential for arguments on the general case in the following section.

\section{Convexity in Higher Dimensions}\label{sec:higher_dim}

We devote this section to discuss the convexity of intersection bodies of polytopes of dimension $d >2$. We make use of the results obtained in \Cref{sec:convexity} to show that, similar to the $2$-dimensional case, the intersection body of a $d$-dimensional parallelepiped is convex if and only if the origin is its center of symmetry. In contrast, we give a sufficient condition under which there are infinitely many positions of the origin for which the intersection body of a given polytope is (strictly) convex.

\begin{remark}\label{rmk:lower-dim-faces}
    To obtain an intersection body $IP$ which is convex, the origin must lie in the interior of $P$. 
    If the origin lies in the interior of a facet, the argument from \Cref{thm:IP_convex_dim2} applies analogously, i.e., $\rho_{IP}(x) = \tfrac{1}{2} \rho_{ I(P \cup -P)}(x)$ for all $x \in \R^2$ except for the two normals of the facet, for which a strict inequality holds. Hence, $\rho_{IP}$ is discontinuous, and therefore $IP$ is not convex. 
    If the origin lies on a lower-dimensional face $F$, there exists a hyperplane $x^\perp$ such that $P \cap x^\perp = F$ and thus the radial function of $IP$ in direction $x$ has value $0$. The set of such $x$ is a cone $V = C \cup -C$, where $C\subset \R^d$ is a convex pointed cone. Then, given $x\in C$, there exist $x_1,x_2 \in \R^d \setminus V$ such that $x$ is a convex combination of $x_1$ and $x_2$. Since $\rho_{IP}(x) = 0$, the segment with extrema $\rho_{IP}(x_1) \: x_1$ and $\rho_{IP}(x_2) \: x_2$ is not entirely contained in the intersection body $IP$, but its extrema are.
\end{remark}

The next result connects the intersection body of a convex body to the intersection body of a prism over the given convex body.

\begin{proposition}\label{prop:prism}
    Let $L \subseteq \R^{d-1}$ be a convex body and $K = L \times [a,b] \subseteq \R^{d-1} \times \R \cong \R^d$ be a prism over $L$. Then, the intersection of $IK$ with the hyperplane $H = \{x \in \R^d \mid x_d=0\}$ is the $(b-a)$th dilate of $IL$, i.e.,
    \[
        IK \cap H = (b-a) \  IL.
    \]
\end{proposition}
\begin{proof}
    Let $u = (\widetilde{u},0) \in H$ and consider its orthogonal complement $u^\perp \subseteq \R^{d}$, which in this case can be interpreted as $\widetilde{u}^\perp \times \R \subseteq \R^{d-1}\times \R$. Then 
    \begin{equation*}
        K \cap u^\perp = (L \times [a,b]) \cap (\widetilde{u}^\perp \times \R) = (L\cap \widetilde{u}^\perp) \times [a,b].
    \end{equation*}
    We can therefore compute the radial function of $IK$ as
    \begin{equation*}
        \rho_{IK}(u) = \vol_{d-1} (K \cap u^\perp) = \vol_{d-1} \big( (L\cap \widetilde{u}^\perp) \times [a,b] \big) = (b-a) \cdot \rho_{IL}(\widetilde{u})
    \end{equation*}
    for $u\in H$. Equivalently, $IK \cap H = (b-a) \  IL$.
\end{proof}
It follows that if $IL$ is non-convex, then so is $IK$. This behavior can be observed in the following example.
\begin{example}
    Consider the unit cube $P = [-1,1]^3$, which is a prism over a square. With the translation $t = (1,1,1)$ we obtain the cube $P+t = [0,2]^3$, and $I(P+t)$ is displayed in \Cref{fig:slice}, from two different points of view. \Cref{prop:prism} implies that $I(P+t) \cap (0,0,1)^\perp$ is the second dilation of the intersection body of the square $[0,2]^2$, which is also displayed at the bottom left of \Cref{fig:affine-lines-square} in red.
    \begin{figure}[ht]
    \centering
    \includegraphics[width=0.36\textwidth]{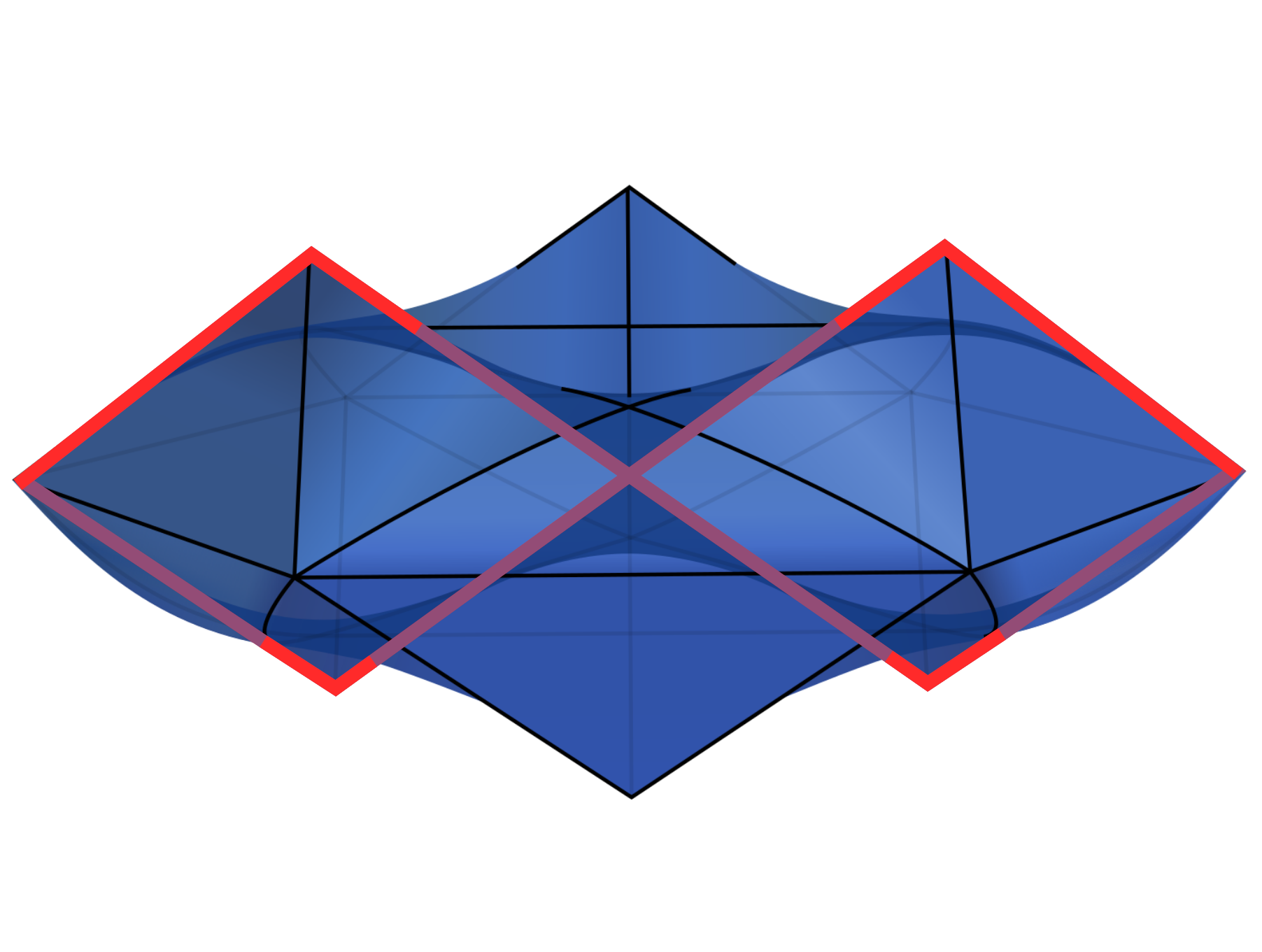}
    \hspace{2em}
    \includegraphics[width=0.3\textwidth]{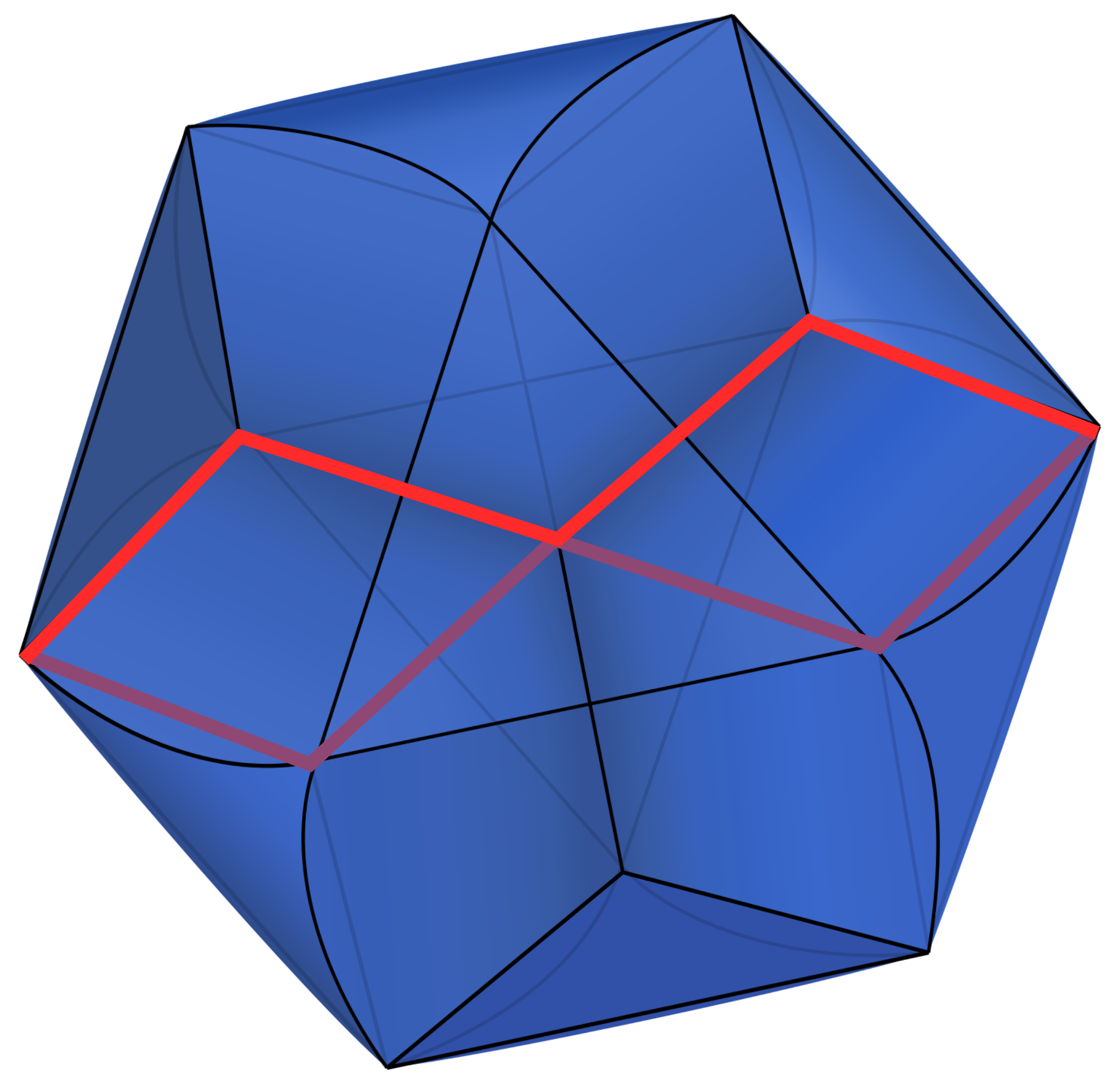}
    \caption{The intersection body of the $3$-dimensional cube $P=[0,2]^3$ (blue) and the intersection body of the square $Q = [0,2]^2$ (red).}
    \label{fig:slice}
    \end{figure}
    
\end{example}

We can now use \Cref{prop:prism} to describe the convexity of intersection body of a parallelepiped in any dimension.
\begin{proposition}\label{prop:cube}
    Let $P = [a_1,b_1]\times [a_2,b_2] \times \dots \times [a_d,b_d]$ be a $d$-dimensional parallelepiped.
    Then $IP$ is convex if and only if $P=-P$.
\end{proposition}
\begin{proof}

     If $P=-P$ then $IP$ is convex by Busemann's Theorem \cite{busemann49_theoremconvexbodies}.
    Conversely, let $P \neq -P$.
    We prove that $IP$ is not convex by induction on $d$. 
    The base case of $d = 2$ follows from \Cref{thm:IP_convex_dim2}.
    Let now $P = [a_1,b_1]\times [a_2,b_2] \times \dots \times [a_d,b_d]$.
    By \Cref{rmk:lower-dim-faces} we assume that the origin lies in the interior of $P$, and thus $a_i < 0 < b_i$ for all $i \in 1,\dots,d$. 
    Without loss of generality, $P \neq -P$ implies that $a_1 \neq -b_1$. Let $Q = P \cap H$, where $H = \{x \in \R^d \mid x_d = 0 \}$. Notice that $P = Q \times [a_d,b_d]$. Thus, $Q$ is a parallelepiped of dimension $(d-1)$ such that $Q \neq -Q$. By induction, this implies that $IQ$ is not convex.  \Cref{prop:prism} implies that $IP \cap H = (b_d - a_d) IQ$. As a consequence, $IP \cap H$ is not convex, and therefore $IP$ is not convex.
\end{proof}

\begin{remark}\label{prop:open-ball}
    We note that whenever the intersection body is strictly convex, then there is an open ball around the origin of translation vectors such that the intersection body is still convex. Indeed, this holds in more generality for the intersection body $IK$ of any star body $K \subseteq \R^d$, with $\origin$ in its interior, and follows directly from the continuity of the volume function, and therefore of the radial function, with respect to $t$. 
    Let $x,y\in \R^d$ and $p_{\star} \!=\! \rho_{I(K+t)}(\star) \cdot\star$ for $\star \in \{x, y, x+y\}$, so that $p_{\star} \in \partial I(K+t)$. Denote by $q_{x+y}$ the point of the segment $\convof{p_x,p_y}$ which is a multiple of $x+y$, namely $q_{x+y} = \frac{\rho_{I(P+t)}(x)\: \rho_{I(K+t)}(y)}{\rho_{I(K+t)}(x) + \rho_{I(K+t)}(y)} (x+y)$. 
    Then, $I(K+t)$ is strictly convex if and only if 
    \begin{equation}\label{eq:convexity}
        \frac{\rho_{I(P+t)}(x)\: \rho_{I(K+t)}(y)}{\rho_{I(K+t)}(x) + \rho_{I(K+t)}(y)} = \frac{\|q_{x+y}\|}{\|x+y\|} < \frac{\|p_{x+y}\|}{\|x+y\|} = \rho_{I(K+t)}(x+y).
    \end{equation}
    This gives a quadratic condition in $\rho_{I(K+t)}$, which is continuous in $t$. Therefore, if \eqref{eq:convexity} holds for $IK$, it holds also for $I(K+t)$ with $t\in B_\varepsilon(\origin)$, for some $\varepsilon>0$.
\end{remark}

The next example shows that strictly convex intersection bodies of polytopes as in \Cref{prop:open-ball} do indeed exist.

\begin{example}\label{ex:strictly-convex}
    The intersection body of the $3$-dimensional centrally symmetric icosahedron $P$ is strictly convex. Indeed, using \texttt{HomotopyContinuation.jl} \cite{HomotopyContinuation.jl} one can check that the algebraic varieties that define the boundary of $IP$ do not contain lines (this is expected, since the generic quintic and sextic surface in $3$-dimensional space do not contain lines). Moreover, because of the central symmetry, the intersection body is convex. Hence, it is strictly convex. This intersection body is displayed in \cite[Figure 1]{BBMS:IntersectionBodiesPolytopes}, and our computations can be verified using the code on MathRepo \cite{mathrepo}.
\end{example}

To summarize, we have studied the admissible positions of the origin with respect to a full-dimensional polytope $P$, such that $IP$ is convex.
For $d=2$ we have shown that the set of admissible positions is precisely the center of symmetry (if it exists).
In higher dimensions it is sometimes infinite, as for the icosahedron, but other times only a single point, as for a cube. We note that proving non-convexity is a much easier task then proving convexity, as the first can be achieved by showing the non-convexity of a small curve on the boundary, while convexity is a global condition. 
A possible approach to tackle this problem in the case of polytopes might be studying the curvature of the algebraic hypersurfaces defining the boundary of the intersection body, as in \cite{BRW:CurvatureHypersurfaces}.

Another interesting direction of research concerns the topology of the set of admissible positions. We collect here some open questions.
\vspace{-0.5em}
\paragraph{Questions:}
\begin{enumerate}
    \item If the set of admissible positions of $P$ is finite, what is its cardinality?
    \item If the set of admissible positions of $P$ is infinite, how many connected components does it have? What is the dimension of these connected components?
    \item If $IP$ is convex but not strictly convex, does this imply $P = -P$?
    \item What are the conditions on $P$ that make $IP$ strictly convex?
\end{enumerate}

\printbibliography

\vspace{0.5cm}
\small{
\subsection*{Affiliations}
\vspace{0.2cm}
\begin{minipage}[t]{0.49\textwidth}
\noindent \textsc{Marie-Charlotte Brandenburg} \\
\textsc{ Max Planck Institute for Mathematics \\ in the Sciences} \\
 \url{marie.brandenburg@mis.mpg.de} \\
 \end{minipage}\hspace{.5em}
 \begin{minipage}[t]{0.49\textwidth}
\noindent \textsc{Chiara Meroni} \\
\textsc{ Max Planck Institute for Mathematics \\ in the Sciences} \\
\url{chiara.meroni@mis.mpg.de}
\end{minipage}
}
\end{document}